\documentclass[11pt,twoside,a4paper]{article}
\usepackage{graphicx}
\usepackage{amsfonts}
\usepackage{bbm}
\usepackage{mathrsfs}
\usepackage{fancyhdr}
 \usepackage{mathrsfs,amsmath,amssymb,amsthm}
\usepackage{color}
 \usepackage{amssymb}
 \usepackage{amsbsy}

\allowdisplaybreaks

 \theoremstyle{plain}
\theoremstyle{remark}  \newtheorem{remark}{\noindent\mbox{Remark}}
 \theoremstyle{plain}
 \theoremstyle{plain}\newtheorem{lemma}{\noindent\mbox{Lemma}}
\theoremstyle{plain} \newtheorem{theorem}{\noindent\mbox{Theorem}}
 \theoremstyle{plain}
 \theoremstyle{plain}
\theoremstyle{definition} \newtheorem{definition}{\noindent\mbox{Definition}}

 \def\qed{\hfill$\Box$\medskip}
 \def\rto{\rightarrow\infty}
\def\z{\left}
\def\y{\right}
 \def\no{\nonumber}

\begin{document}
 \title{\textbf{Regeneration of branching processes with immigration in varying environments}\thanks{Supported by National
Natural Science Foundation of China (Grant No. 11601494;11501008)}}                  

\author{   Baozhi \uppercase{Li}$^\dag$ \and Hongyan  \uppercase{Sun}$^\ddag$ \and
    Hua-Ming \uppercase{Wang}$^{\dag,\S}$
}
\date{}
\maketitle%
 \footnotetext[2]{School of Mathematics and Statistics, Anhui Normal University, Wuhu 241003, China  }
\footnotetext[3]{School of Sciences, China University of Geosciences, Beijing 100083, China}
\footnotetext[4]{Email: hmking@ahnu.edu.cn}
\vspace{-.5cm}

\begin{center}
\begin{minipage}[c]{12cm}
\begin{center}\textbf{Abstract}\quad \end{center}
In this paper, we consider certain linear-fractional branching processes with immigration in varying environments. For $n\ge0,$ let $Z_n$ counts the number of individuals of the $n$-th generation, which excludes the immigrant which enters into the system at time $n.$  We call $n$ a regeneration time if $Z_n=0.$ We give first a criterion for the finiteness or infiniteness of the number of regeneration times.  Then, we construct some concrete examples to exhibit the strange phenomena caused by the so-called varying environments. It may happen that the process is extinct but there are only finitely many regeneration times. Also, when there are infinitely many regeneration times, we show that for each $\varepsilon>0,$ the number of regeneration times in $[0,n]$ is no more than  $(\log n)^{1+\varepsilon}$ as $n\rto.$

\vspace{0.2cm}

\textbf{Keywords:}\  Branching processes,  varying environments, immigration, regeneration
\vspace{0.2cm}

\textbf{MSC 2020:}\ 60J80, 60J10
\end{minipage}
\end{center}

\section{Introduction}
\subsection{Background and motivation}
It is known that Galton-Watson processes are widely applied in nuclear physics, biology, ecology, epidemiology and many others and have been extensively studied, see \cite{an72,hjv,ka} and references therein. The study of Galton-Watson processes can be extended directly in two directions. One popular extension is the branching processes in random environment(BPRE hereafter). This object has attracted the attentions of many authors and its development is much satisfying up to now. Many interesting results arise from the existence of the random environments.  We refer the reader to \cite{kv} and references therein for details. Another interesting extension of the Galton-Watson process is the branching processes in varying environment(BPVE hereafter). Compared with BPREs, the situation of the study of BPVEs is not so satisfying. The main reason is that a BPVE is not a time homogeneous Markov chain any longer. But BPREs do own some homogeneous property. Indeed, if the environments are assumed to be stationary and ergodic, then a BPRE is a time homogeneous process under the annealed probability.
The emerging of the so-called varying environments also brings some strange phenomena to the branching processes. For example, the process may ``fall asleep" at some positive state \cite{lin}, it may  diverge at different exponential rates \cite{ms} and the tail probabilities of the surviving time may show some strange asymptotics \cite{fuj, wy}. For other aspects of the study of BPVEs, we refer the reader to \cite{bp, bcn, cw, dhkp, j, jon, ker} and references therein.

In this paper, we study BPVEs with immigration. For simplicity, we assume that in each generation, only one immigrant immigrates into the system. Roughly speaking,  for $n\ge0,$ let $Z_n$ be the population size of individuals in the $n$-th generation, which does not count the immigrant entering the system at time $n.$ If $Z_n=0,$ we call $n$ a regeneration time. Our aim is to give necessary and sufficient conditions to decide whether the process has finite or infinitely many regeneration times and study the asymptotics of  the number of regeneration times in the interval $[0,n]$ when the process owns infinitely many regeneration times. We should note that for Galton-Watson processes or BPREs with immigration, if the process has one regeneration time, it must has infinitely many regeneration times, in other words, it will never happen that such a process owns finitely many regeneration times if there are any due to the time homogeneity. But for BPVEs, we construct some concrete near-critical BPVEs with immigration, which exhibit some totally different phenomena. On one side, it may happen that the process is extinct but there are only finitely many regeneration times. On the other side, for such near-critical settings, we also show that for each $\varepsilon>0,$ the number of regeneration times in $[0,n]$ is no more than  $(\log n)^{1+\varepsilon}$ as $n\rto.$

Our motivation originates from two aspects, the regeneration structure of BPREs with immigration and the cutpoints of random walks in varying environments. On one side, in \cite{kks}, in order to study the stable limit law of random walks in random environments, a regeneration structure of a single-type BPRE with immigration was constructed and the tail probabilities of the regeneration time and the number of the total progeny before the first regeneration time were estimated further.   The related problems of the multitype case of this regeneration structure can be found in \cite{key, roi, wact}. Along this line, it is natural for us to consider the number of regeneration times of BPVEs with immigration. On the other side, in \cite{cfrb, jlp, lmw, wmp}, a class of questions related to the cutpoints of the random walks in varying environments was considered. We find that the regeneration structures for BPVEs with immigration and the excursions between successive cutpoints share some similarities, so that we aim to study the regeneration of BPVEs with immigration in this paper.

We treat currently only the regeneration times of one-type  BPVEs with immigration with geometric offspring distributions. Our method also works for the two-type case which is much complicated and will appear in another forthcoming paper.

Before giving the model and stating the main results, we introduce here some conventions and notations which will be used in what follows. We use the notation $\#\{\ \}$ to count the number of elements in a set $\{\ \}.$ The notation $a(n)\sim b(n)$ means $a(n)/b(n)\rightarrow 1$ as $n\rto.$ Unless otherwise stated, $c$ is a strictly positive number whose value may change from one to another. We also adopt the convention that empty product equals $1$ and empty sum equals $0$.

\subsection{Models and main results}

For $k\ge1,$ suppose $0<p_k\le 1/2,q_k>0$  are numbers such that  $p_k+q_k=1$ and
 $$f_k(s)=\frac{p_k}{1-q_ks}, s\in [0,1].$$
Let $Z_n,n\ge0$ be a Markov Chain such that $Z_0=0$ and
\begin{align}
  E\z( s^{Z_n}\big|Z_0,...,Z_{n-1}\y)=\z[f_{n}(s)\y]^{1+Z_{n-1}}, n\ge1.\label{dz}
\end{align}
Clearly,   $\{Z_n\}_{n\ge0}$ forms a branching process in varying environment  with exactly one immigrant in each generation.
Now, we define the regeneration time which we are concerned.

\begin{definition}
   Let $C=\{n\ge0: Z_n=0\}$ and for $k\ge1,$ let $C_k=\{n: n+i\in C, 0\le i\le k-1\}.$ If $n\in C,$ we call $n$ a regeneration time of the process $\{Z_n\}.$ If $n\in C_k,$ we call $n$ an $k$-strong regeneration time  of the process $\{Z_n\}.$
\end{definition}
For $k\ge1,$ let $m_k=f_k'(1)=\frac{q_k}{p_k},$ and for $n\ge k\ge 1,$ set
\begin{align}
  D(k,n):=1+\sum_{j=k}^n m_j\cdots m_n\text{ and write simply }D(n)\equiv D(1,n). \label{dkn}
\end{align}
We are now ready to state the main results.
\begin{theorem}\label{m} Suppose that $p_n=1/2-r_n, n\ge1$ where $0\le r_n\le 1/2-\varepsilon$ for some $\varepsilon>0.$ Let $D(n), n\ge1$ be the ones in \eqref{dkn}. If $$\sum_{n=2}^\infty\frac{1}{D(n)\log n}<\infty,$$ then $\{Z_n\}$ has at most finitely many regeneration times almost surely. If there exists some $\delta>0$ such that $D(n)\le \delta n\log n$ for $n$ large enough and $$\sum_{n=2}^\infty\frac{1}{D(n)\log n}=\infty,$$ then $\{Z_n\}$ has infinitely many $k$-strong regeneration times almost surely.
\end{theorem}
Next we consider some near-critical BPVEs. Fix $B\ge 0$ and let $i_0$ be a positive number such that $\frac{B}{4i_0}<\frac{1}{2}.$ For $i\ge1,$
set \begin{align}\label{pi}
  p_i=\z\{\begin{array}{ll}
    {1}/{2}-\frac{B}{4i}, &i >i_0,\\
    {1}/{2}, &i\le i_0.
  \end{array}\y.
\end{align}
 We have the following criterion for the finiteness of the number of regeneration times.
\begin{theorem}\label{nc}
  Fix $B\ge0$  and and for $i\ge1,$ let $p_i$ be the one in \eqref{pi}.
  If $B\ge1,$ then $\{Z_n\}$ has at most finitely many regeneration times almost surely. Otherwise, if $B<1,$ then $\{Z_n\}$ has infinitely many regeneration times almost surely.
\end{theorem}
\begin{remark}
  Let $\{Y_n\}$ be a BPVE such that $Y_0=1$ and
  $E\z( s^{Y_n}\big|Y_0,...,Y_{n-1}\y)=\z[f_{n}(s)\y]^{Y_{n-1}}$ for $n\ge1.$
  Denote by $\nu=\inf\{n\ge 0: Y_n=0\}$ be the extinction time of $\{Y_n\}.$ Using \eqref{fkn}, we get $P(\nu>n)=1/(1+\sum_{j=1}^n m_1^{-1}\cdots m_{j}^{-1}).$
  If $B=1,$ then from \eqref{sms}, we have $\sum_{i=1}^{n}m_1^{-1}\cdots m_i^{-1}\sim c\log n\rightarrow\infty$  as $n\rto.$ Therefore, $P(\nu=\infty)=0,$ that is,  with probability 1, $\{Y_n\}$ is extinct. But in this case,  we see from Theorem \ref{nc} that
$\{Z_n\}$ has at most finitely many regeneration times. Such phenomenon never happens
to time-homogeneous branching process. Indeed, whenever $m_k\equiv m, k\ge1,$ it is clear that if $\{Y_n\}$ is extinct, then $\{Z_n\}$ must have infinitely many regeneration times.


\end{remark}

When there are infinitely many regeneration times, the following theorem gives the asymptotics of the number of regeneration times in $[0,n].$
\begin{theorem}\label{nr}
  Fix $0\le B<1$  and for $i\ge 1,$ let $p_i$ be the one in \eqref{pi}. Then
   \begin{align*}
        & \lim_{n\rto}\frac{E\#\{k: k\in C\cap[0,n]\}}{\log n}=c>0,\end{align*}
and for all $\varepsilon>0,$
   \begin{align*}
    &\lim_{n\rto}\frac{\#\{k: k\in C\cap[0,n]\}}{(\log n)^{1+\varepsilon}}=0,  \text{ a.s..}
  \end{align*}
  \end{theorem}
\begin{remark}
  Notice that Theorem \ref{nr} contains the case $B=0.$ In this case, $p_i\equiv 1/2$ and $m_i\equiv 1$ for all $i\ge1,$ so that $\{Z_n\}$ is indeed a critical Galton-Watson process with immigration. As what Theorem \ref{nr} shows, it seems that up to the multiplication of a positive constant, the value of $0\le B<1$ does not affect the order of the number of regeneration times in $[0,n].$
\end{remark}
\noindent{\bf Outline of the paper.} The remainder of the paper is arranged as follows.
In Section \ref{sec2}, we give some auxiliary results. The detailed proofs of the main results will be given in Section \ref{sec3}.

\section{Auxiliary results}\label{sec2}
In this section, we give some preliminary results which will be used for proving the main theorems. To begin with, let $D(n)$ and $D(k,n), n\ge k\ge1$ be those defined in \eqref{dkn}.  By some easy computation, we get
\begin{align}
  D(n+1)&=1+m_{n+1}D(n),\  \frac{D(k,n)}{D(n)}=1-\prod_{j=k-1}^n\z(1-\frac{1}{D(j)}\y).\label{dnp}
\end{align}
Clearly, if $m_i\ge1$ for all $i\ge1,$ then $D(n)$ is monotone increasing in $n.$

Next,  we consider the probability generating function of the population size. For $n\ge0,$ let
 $F_n(s)=E(s^{Z_n}),s\in [0,1].$ By induction, from \eqref{dz}  we get
\begin{align}\label{fn}
  F_n(s)=\prod_{k=1}^n f_{k,n}(s), s\in [0,1], n\ge0,
\end{align}
where
$f_{k,n}(s)=f_k(f_{k+1}(\cdots(f_n(s)) )), n\ge k\ge 1.$
Recall that $m_k=f_k'(1)=\frac{q_k}{p_k},k\ge1.$ Then it is easily seen that
$f_k(s)=1-\frac{m_k(1-s)}{1+m_k(1-s)}, k\ge1.$ As a consequence, it follows by induction that
\begin{align}
  f_{k,n}(s)&=1-\frac{m_k\cdots m_n(1-s)}{1+\sum_{j=k}^n m_j\cdots m_n(1-s)}
  =\frac{1+\sum_{j=k+1}^n m_j\cdots m_n(1-s)}{1+\sum_{j=k}^n m_j\cdots m_n(1-s)}.\label{fkn}
\end{align}
Substituting \eqref{fkn} into \eqref{fn}, we get
\begin{align*}
  F_n(s)=\frac{1}{1+\sum_{j=1}^n m_j\cdots m_n(1-s)}
\end{align*}
which leads to
\begin{align}\label{zn0}
  P(Z_n=0)&=F_n(0)=\frac{1}{1+\sum_{j=1}^n m_j\cdots m_n}=\frac{1}{D(n)},n\ge1.
\end{align}
Similarly, we have
\begin{align}\label{zkn0}P(Z_n=0|Z_k=0)=\frac{1}{D(k,n)},n\ge k\ge1.\end{align}
With \eqref{zn0} and \eqref{zkn0} in hands, we can show the following lemma.
\begin{lemma}\label{prs}  Fix $k\ge1.$ {\rm(i)} For $n\ge1$ we have $$P(n\in C_k)=\frac{\prod_{i=1}^{k-1}p_{n+i}}{D(n)}.$$ {\rm (ii)} For $l\ge n+k,$ we have $$P(n\in C_k, l\in C_k)=\frac{\prod_{i=1}^{k-1}p_{n+i}}{D(n)}\frac{\prod_{i=1}^{k-1}p_{l+i}}{D(n+k,l)}.$$
\end{lemma}
\begin{proof}
  Fix $k\ge1.$ Using \eqref{zn0}, by definition, we have
  \begin{align*}
    P(n\in C_k)&=P(Z_n=0,Z_{n+1}=0,...,Z_{n+k-1}=0)\\
    &=P(Z_n=0)\prod_{j=0}^{k-2}P(Z_{n+j+1}|Z_{n+j}=0)\\
    &=\frac{\prod_{i=1}^{k-1}p_{n+i}}{D(n)}.
  \end{align*}
  The first part of the lemma is proved.
  To prove the second part, note that similar to first part, using \eqref{zkn0}, we have
  $$P(j\in C_k|Z_i=0)=\frac{\prod_{s=1}^{k-1}p_{j+s}}{D(i+1,j)}.$$ Consequently, we get
  \begin{align*}
    P(n&\in C_k, l\in C_k)=P(n\in C_k)P(l\in C_k|n\in C_k)\\
    &=P(n\in C_k)P(l\in C_k|Z_{n+k-1}=0)\\
    &=\frac{\prod_{i=1}^{k-1}p_{n+i}}{D(n)}\frac{\prod_{i=1}^{k-1}p_{l+i}}{D(n+k,l)},
  \end{align*}
  which finishes the proof of the second part. \qed
\end{proof}
\section{Proofs}\label{sec3}

\subsection{Proof of Theorem \ref{m}}
 In order to prove Theorem \ref{m}, we adopt an approach similar to the one used in \cite{cfrb}. To begin with, we prove the first part.
 For $j<i,$ set  $C_{j,i}=\{(2^j, 2^i]: x\in C\}$ and let $A_{j,i}=|C_{j,i}|$ be the cardinality of the set $C_{j,i}.$ On the event $\{A_{m,m+1}>0\},$ let $l_m=\max\{k:k\in C_{m,m+1}\}$ be the largest regeneration time in $C_{m,m+1}.$
Then for $m\ge1,$ we have
\begin{align}
  &\sum_{j=2^{m-1}+1}^{2^{m+1}}P(j\in C)=E(A_{m-1,m+1})\no\\
    &\quad\quad\ge \sum_{n=2^m+1}^{2^{m+1}}E(A_{m-1,m+1}, A_{m,m+1}>0,l_m=n)\no\\
    &\quad\quad=\sum_{n=2^m+1}^{2^{m+1}}P(A_{m,m+1}>0,l_m=n)E(A_{m-1,m+1}| A_{m,m+1}>0,l_m=n)\no\\
    &\quad\quad=\sum_{n=2^m+1}^{2^{m+1}}P(A_{m,m+1}>0,l_m=n)\sum_{i=2^{m-1}+1}^nP(i\in C| A_{m,m+1}>0,l_m=n)\no\\
    &\quad\quad\ge P(A_{m,m+1}>0)\min_{2^m<n\le 2^{m+1}}\sum_{i=2^{m-1}+1}^nP(i\in C| A_{m,m+1}>0,l_m=n)\no\\
    &\quad\quad=:a_mb_m. \label{smpjc}
\end{align}
Fix $2^{m}+1\le n\le 2^{m+1}$ and $2^{m-1}+1\le i\le n.$ Using Lemma \ref{prs} and the Markov property, we get
\begin{align}
  P(&i\in C| A_{m,m+1}>0,l_m=n)\no\\
  &=\frac{P\z(Z_i=0, Z_n=0, Z_t\ne0, n+1\le t\le 2^{m+1}\y)}{P\z(Z_n=0, Z_t\ne0, n+1\le t\le 2^{m+1}\y)}\no\\
  &=\frac{P(Z_i=0,Z_n=0)}{P(Z_n=0)}\frac{P\z( Z_t\ne0, n+1\le t\le 2^{m+1}|Z_i=0, Z_n=0\y)}{P\z(Z_t\ne0, n+1\le t\le 2^{m+1}|Z_n=0\y)}\no\\
  &= \frac{P(Z_i=0,Z_n=0)}{P(Z_n=0)}=\frac{D(n)}{D(i)D(i+1,n)}.\label{nl}
\end{align}
  But since $m_j\ge 1$ for all $j\ge1,$ we have
  \begin{align}
    &\frac{D(n)}{D(i)D(i+1,n)}=\frac{\sum_{j=1}^{n+1}m_j\cdots m_n}{\sum_{j=1}^{i+1}m_j\cdots m_i\sum_{j=i+1}^{n+1}m_j\cdots m_n}\no\\
    &\quad\quad=\frac{\sum_{j=1}^{n+1}m_1^{-1}\cdots m_{j-1}^{-1}}{\sum_{j=1}^{i+1}m_1^{-1}\cdots m_{j-1}^{-1}\sum_{j=i+1}^{n+1}m_{i+1}^{-1}\cdots m_{j-1}^{-1}}\no\\
    &\quad\quad\ge \frac{1}{\sum_{j=i+1}^{n+1}m_{i+1}^{-1}\cdots m_{j-1}^{-1}}\ge\frac{1}{n-i+1}.\label{pd}   \end{align}
Thus, taking \eqref{nl} and \eqref{pd} together, we deduce that
\begin{align}
  b_m&=\min_{2^m<n\le 2^{m+1}}\sum_{i=2^{m-1}+1}^nP(i\in C| A_{m,m+1}>0,l_m=n)\no\\
  &\ge \min_{2^m<n\le 2^{m+1}}\sum_{i=2^{m-1}+1}^n \frac{1}{n-i+1}=\min_{2^m<n\le 2^{m+1}}\sum_{j=1}^{n-2^{m-1}} \frac{1}{j}\no\\
  &=\sum_{j=1}^{2^{m-1}+1} \frac{1}{j}\ge \int_{1}^{2^{m-1}+2}\frac{1}{x}dx\ge (m-1)\log 2. \label{ebm}
\end{align}
Substituting \eqref{ebm} into \eqref{smpjc} and using Lemma \ref{prs}, we see that
\begin{align}
  \sum_{m=1}^{\infty}&P(A_{m,m+1}>0)\le  \sum_{m=1}^{\infty}\frac{1}{b_m}\sum_{j=2^{m-1}+1}^{2^{m+1}}P(j\in C)\no\\
  &\le c\sum_{m=1}^{\infty} \frac{1}{m}\sum_{j=2^{m-1}+1}^{2^{m+1}}\frac{1}{D(j)}\le c\sum_{m=1}^{\infty} \sum_{j=2^{m-1}+1}^{2^{m+1}}\frac{1}{D(j)\log j}\no\\
  &\le c\sum_{n=2}^\infty \frac{1}{D(n)\log n}.\no
\end{align}
Therefore, if $\sum_{n=2}^\infty \frac{1}{D(n)\log n}<\infty,$ then it follows by the Borel-Cantelli lemma that with probability 1, at most finitely many of the events $\{A_{m,m+1}>0\},m\ge1$ occur. Consequently, the process $\{Z_n\}$ has at most finitely many regeneration times. The first part of Theorem \ref{m} is proved.

Next we turn to prove the second part. Suppose there exists some $\delta>0$ such that $D(n)\le \delta n\log n$ for $n$ large enough and $\sum_{n=2}^\infty\frac{1}{D(n)\log n}=\infty.$

For $j\ge1,$ let $n_j=[j\log j]$ be the integer part of $j\log j$ and set $A_j=\{n_j\in C_k\}.$
In what follows, we fix a integer $j_0>0$ such that $(j_0+1)\log (j_0+1)-j_0\log j_0>k.$
By Lemma \ref{prs}, it follows that
\begin{align}
\sum_{j=j_0}^{\infty}P(A_j)=\sum_{j=2}^{\infty}\frac{\prod_{i=1}^{k-1}p_{n_j+i}}{D(n_j)}\ge c \sum_{j=j_0}^{\infty}\frac{1}{D([j\log j])}. \label{ad}
\end{align}
Since $m_j\ge 1$ for all $j\ge1,$ from \eqref{dnp} we see that $D(n)$ is increasing in $n.$
Thus applying \cite[Lemma 2.2]{cfrb}, we conclude that
$\sum_{j=j_0}^{\infty}\frac{1}{D([j\log j])}$ and $\sum_{j=j_0}^{\infty}\frac{1}{D(j)\log j}$  converge or diverge simultaneously.  Therefore, it follows from \eqref{ad} that
\begin{align}
  \sum_{j=j_0}^{\infty}P(A_j)=\infty.\label{spa}
\end{align}

For $j_0\le j<l,$  using \eqref{dnp} and Lemma \ref{prs}, we abtain
\begin{align}
  P(A_jA_l)&=P(n_j\in C_k, n_l\in C_k)=\frac{\prod_{i=1}^{k-1}p_{n_j+i}}{D(n_j)}\frac{\prod_{i=1}^{k-1}p_{n_l+i}}{D(n_j+k,n_l)}\no\\
  &=P(A_j)P(A_l) \frac{D(n_l)}{D(n_j+k,n_l)}\no\\
  &=P(A_j)P(A_l) \z(1-\prod_{i=n_j+k-1}^{n_l}\z(1-\frac{1}{D(i)}\y)\y)^{-1}\no\\
  &\le P(A_j)P(A_l)\z(1-e^{-\sum_{i=n_j+k-1}^{n_l}\frac{1}{D(i)}}\y)^{-1}.\label{ajl}
  \end{align}
Fix $\varepsilon>0$ and $j\ge j_0.$ Let $$\ell=\min\z\{l\ge j: \sum_{i=n_j+k-1}^{n_l}\frac{1}{D(i)}\ge \log \frac{1+\varepsilon}{\varepsilon}\y\}.$$
Clearly, for $l\ge\ell,$ $\z(1-e^{-\sum_{i=n_j+k-1}^{n_l}\frac{1}{D(i)}}\y)^{-1}\le 1+\varepsilon.$
Thus it follows from \eqref{ajl} that
\begin{align}
  P(A_jA_l)\le (1+\varepsilon)P(A_jA_l), \forall l\ge\ell. \label{up1}
\end{align}
Suppose next $j<l<\ell.$ For $0<u<\log \frac{1+\varepsilon}{\varepsilon},$ we have $1-e^{-u}\ge cu$ for some $c(\varepsilon)>0$ small enough. This fact and \eqref{ajl} yield that
\begin{align}
    P(A_jA_l)&\le cP(A_j)P(A_l)\z(\sum_{i=n_j+k-1}^{n_l}\frac{1}{D(i)}\y)^{-1}\no\\
    &\le c\frac{D(n_l)}{n_l-n_j+2}P(A_j)P(A_l)=c\z(\prod_{i=1}^{k-1}p_{n_l+i}\y)\frac{P(A_j)}{n_l-n_j-k+2}\no\\
    &\le\frac{cP(A_j)}{l\log l-j\log j},\no
  \end{align}
where for the second inequality we use the fact that $D(n)$ is increasing in $n.$
Consequently,
\begin{align}
    \sum_{j<l< \ell}P(A_jA_l)&\le \sum_{j<l< \ell }\frac{cP(A_j)}{l\log l-j\log j}\le cP(A_j)\sum_{l=j+1}^{\ell-1}\frac{1}{l\log l-k\log k}\no\\
    &\le cP(A_j)\frac{1}{\log j}\sum_{l=j+1}^{\ell-1}\frac{1}{l-j}\le cP(A_j)\frac{\log \ell}{\log j}.\label{eu}
    \end{align}
    Recall that \begin{align}
      \sum_{i=n_j+k-1}^{n_l}\frac{1}{D(i)}< \log \frac{1+\varepsilon}{\varepsilon}, j<l<\ell \text{ and } D(n)\le \delta n\log n \label{cd}
    \end{align}
    for some $\delta>0$ and $n$ large enough.
    We claim that if $j$ is large enough then \begin{align}
      \ell\le j^\gamma\text{ if } \gamma>\z(\frac{1+\varepsilon}{\varepsilon}\y)^{\delta/c}+\varepsilon.\label{cc}
    \end{align}
    Suppose in contrary that $\ell > j^\gamma.$ Then  for $j$ large enough
    \begin{align}
      \sum_{i=n_j+k-1}^{n_{\ell}}&\frac{1}{D(i)}\ge \frac{1}{\delta} \sum_{i=n_j+k-1}^{n_{\ell}}\frac{1}{i\log i}\ge \frac{1}{\delta}(\log\log n_{\ell}-\log\log (n_j+k-1 ))\no\\
      &\ge \frac{c}{\delta}(\log\log n_{\ell}-\log\log n_j) =  \frac{c}{\delta} \log \frac{\gamma\log j+\log\gamma+\log\log j}{\log j+\log\log j}   \no\\
      &\ge \frac{c}{\delta}\log(\gamma -\varepsilon)\ge \log \frac{1+\varepsilon}{\varepsilon},\no
    \end{align}
    which contradicts \eqref{cd}.

    Substituting \eqref{cc} into \eqref{eu}, for $j$ large enough, we have
    \begin{align}
      \sum_{j<l< \ell}&P(A_jA_l)\le  cP(A_j).\label{up2}
    \end{align}
    Taking \eqref{up1} and \eqref{up2} together, we conclude that for some integer $j_1\ge j_0>0,$
    \begin{align}
    \sum_{j=j_1}^N\sum_{j<l\le N}P(A_jA_l)\le \sum_{j=j_1}^N\sum_{j<l\le N, }(1+\varepsilon)P(A_j)P(A_l)+c\sum_{j=j_0}^N P(A_j)\no.
    \end{align}
    Therefore, taking \eqref{spa} into account, we have
   \begin{align}
     \alpha&:=\varliminf_{N\rto}\frac{\sum_{j=j_1}^N\sum_{j<l\le N}P(A_jA_l)-\sum_{j=j_1}^N\sum_{j<l\le N}(1+\varepsilon)P(A_j)P(A_l)}{\z(\sum_{j=j_1}^N P(A_j)\y)^2}\no\\
     &\le \varliminf_{N\rto} \frac{c}{\sum_{j=j_0}^N P(A_j)}=0.\no
   \end{align}
   An application of Borel-Cantelli lemma (see \cite{pe04}, p.235) yields that
   \begin{align}
     P(A_j,j\ge j_1\text{ occur infinitely often})\ge \frac{1}{1+\varepsilon +2\alpha}\ge \frac{1}{1+\varepsilon}.\no
   \end{align}
   Since $\varepsilon>0$ is arbitrary, we come to the conclusion that
    \begin{align}
     P(A_j,j\ge j_1\text{ occur infinitely often})=1.\no
   \end{align}
   The second part of the theorem is proved. \qed

\subsection{Proof of Theorem \ref{nc}} Fix $B\ge0$ and for $i\ge 1,$ let $p_i$ be the one in \eqref{pi}. Then $$m_i=\frac{q_i}{p_i}=\frac{\frac{1}{2}+\frac{B}{4i}}{\frac{1}{2}-\frac{B}{4i}}=1+\frac{B}{i}-\frac{B}{2i^2}+o\z(\frac{1}{i^2}\y)$$ as $i\rto.$ Consequently, we get
\begin{align}\label{ms}
  m_1\cdots m_n\sim cn^B, \text{ as }n\rto,
\end{align} which implies that
\begin{align}\label{sms}
  \sum_{i=1}^{n}m_1^{-1}\cdots m_i^{-1}\sim \left\{\begin{array}{ll}
    c,&\text{ if }B>1,\\
    c\log n,& \text{ if }B=1,\\
    c(1-B)^{-1}n^{1-B},&\text{ if }B<1,
  \end{array}\right. \text{ as }n\rto.
\end{align}
As a result, we have
\begin{align}\label{dln}
  \frac1{D(n)\log n}=\frac{m_1^{-1}\cdots m_n^{-1}}{\log n\sum_{i=1}^{n+1}m_1^{-1}\cdots m_{i-1}^{-1}}\sim \left\{\begin{array}{ll}
    \frac{c}{n^{B}\log n},&\text{ if }B>1,\\
    \frac{c}{n(\log n)^2},& \text{ if }B=1,\\
    \frac{c}{n\log n},&\text{ if }B<1,
  \end{array}\right.
\end{align}
as $n\rto.$
We thus come to the conclusion that
\begin{align}
  \sum_{n=2}^\infty\frac1{D(n)\log n} \z\{\begin{array}{ll}
<\infty&\text{ if } B\ge1,\\
=\infty &    \text{ if } B<1.
  \end{array}\y.\no
\end{align}
Note also that if $B<1,$ then from \eqref{dln} we see that $D(n)\sim cn$ as $n\rto.$ Therefore, applying Theorem \ref{m}, we finish the proof of Theorem \ref{nc}. \qed

\subsection{Proof of Theorem \ref{nr}} For $n\ge0,$ let $$X_n=\left\{\begin{array}{cc}
                         1 & \text{ if } Z_n=0, \\
                         0 & \text{ if } Z_n>0.
                       \end{array}
\right.$$
and  set $S_n=\sum_{k=0}^{n}X_i.$ It is easy to see that
$\#\{k: k\in C\cap[0,n]\}\equiv S_n.$

Since $0\le B<1,$ from \eqref{ms} and \eqref{sms}, we get  \begin{align}
  &m_1\cdots m_n\sim cn^B \text{ and }  \sum_{i=1}^{n}m_1^{-1}\cdots m_i^{-1}\sim c n^{1-B}, \text{ as }n\rto.\no
\end{align}
Consequently, from \eqref{zn0}, we have
\begin{align}\label{px}
  P(X_n=1)&=P(Z_n=0)=\frac{1}{1+\sum_{j=1}^n m_j\cdots m_n}\no\\
  &=\frac{m_1^{-1}\cdots m_n^{-1}}{1+\sum_{j=1}^n m_1^{-1}\cdots m_j^{-1}}\sim c\frac{n^{-B}}{n^{1-B}}=\frac{c}{n}, \text{ as }n\rto.
\end{align}
Therefore,
$$ ES_n\sim c\sum_{k=1}^n \frac{1}{k}\sim \log n, \text{ as }n\rto. $$
The first part of Theorem \ref{nr} is proved.

To prove the second part, noticing that
 $S_n$ is nonnegative and nondecreasing in $n,$ thus by \eqref{px}, we must have
\begin{align}\label{ss}
  E(\max_{1\le l\le n}|S_l|)=E(S_n)\le \sum_{i=1}^{n}\frac{c}{i}.
\end{align}
But for each $\varepsilon>0,$
\begin{align}\label{bc}
  \sum_{i=1}^{n}\frac{1}{i(\log i)^{(1+\varepsilon)}}<\infty.
\end{align}
Therefore, with \eqref{ss} and \eqref{bc} in hands,  applying \cite[Theorem 2.1]{fk01}, we conclude that
\begin{align*}
  \lim_{n\rto}\frac{\#\{k: k\in C \cap[0,n]\}}{(\log n)^{1+\varepsilon}}=\lim_{n\rto}\frac{S_n}{(\log n)^{1+\varepsilon}}=0, \text{ a.s..}
\end{align*}
The theorem is proved. \qed.


\vspace{.5cm}

\noindent{{\bf \Large Acknowledgements:}} The authors would like to thank Miss Tang Lanlan for some useful discussions  when writing the paper.


\begin{thebibliography}{99}
\addtolength{\itemsep}{-0.5em}

%

\bibitem{an72}{\sc Athreya, K. B. and  Ney, P. E.} (1972). Branching processes. Springer-Verlag.


\bibitem{bp} {\sc Bhattacharya, N. and Perlman, M.} (2017).  Time inhomogeneous branching processes conditioned on non-extinction. Preprint. arXiv:1703.00337 [math.PR].

\bibitem{bcn} {\sc Biggins, J. D., Cohn, H. and Nerman, O.} (1999). Multi-type branching in varying environment. {\em Stoch. Proc. Appl.} {\bf83,} 357--400.

\bibitem{cw} {\sc Cohn, H. and Wang, Q.} (2003). Multitype branching limit behavior. {\em Ann. Appl. Probab.} {\bf13,} 490--500.

\bibitem{cfrb} {\sc Cs\'aki, E.   F\"oldes, A. and R\'ev\'esz, P.} (2010). On the number of cutpoints of the transient nearest neighbor random walk on the line. {\it J. Theor. Probab.,}  {\bf 23}(2), 624-638.

\bibitem{dhkp} {\sc Dolgopyat, D., Hebbar, P., Koralov, L. and Perlman, M.} (2018). Multi-type branching processes with time-dependent branching rates. {\em J. Appl. Probab.}  {\bf55,}  701--727.


   \bibitem{fk01}{\sc Fazekas, I. and Klesov,  O.} (2001). A general approach to the strong law of large numbers. {\em Theor.  Probab.  Appl.,} {\bf 45}(3), 436--449.

\bibitem{fuj} {\sc Fujimagari, T.} (1980).   On the extinction time distribution of a branching process in varying environments. {\em Adv. Appl. Probab.}  {\bf12,}  350--366.

\bibitem{j} {\sc Jagers, P.} (1974).  Galton-Watson processes in varying environment. {\em J. Appl. Probab.} {\bf11,} 174--178.

\bibitem{hjv}{\sc  Haccou, P. Jagers, P.  and  Vatutin, V. A.} (2005). Branching processes: variation, growth, and extinction of populations. Cambridge University Press, New York.

 \bibitem{jlp}{\sc James, N., Lyons, R. and Peres, Y.} (2008). A transient Markov chain with finitely many cutpoints. In: {\it IMS Collections Probability and Statistics: Essays in Honor of David A. Freedman}  {\bf2}, 24-29. Institute of Mathematical Statistics.

\bibitem{jon}  {\sc Jones, O. D.} (1997). On the convergence of multitype branching processes with varying environments. {\em Ann. Appl. Probab.} {\bf7,} 772--801.

\bibitem{ker} {\sc Kersting, G.} (2020). A unifying approach to branching processes in a varying environment. {\em J. Appl. Probab.} {\bf57,} 196--220.

\bibitem{kv} {\sc Kersting, G. and Vatutin, V.} (2017). {\em Discrete time branching processes in random environment.} John Wiley \& Sons, Inc., USA.

\bibitem{kks}{\sc Kesten, H., Kozlov, M. V. and Spitzer, F.} (1975). A limit law for random walk in a random environment. {\it Compos. Math.}  {\bf30}, 145-168.

\bibitem{key} {\sc Key, E. S.} (1987). Limiting distributions and regeneration times for multitype branching processes with immigration in a random environment. {\it Ann. Probab.}  {\bf15}(1), 344-353.

\bibitem{ka} {\sc Kimmel, M. and  Axelrod, D. E.} (2015). Branching processes in biology. Springer-NewYork.

\bibitem{lin} {\sc Lindvall, T.} (1974). Almost sure convergence of branching processes in varying and random environments. {\em Ann. Probab.} {\bf2,} 344-346.

\bibitem{lmw}{\sc Lo, C. H., Menshikov, M. V. and Wade, A. R.} (2020). Cutpoints of non-homogeneous random walks.  Preprint. arXiv: 2003.01684. [math.PR].
%
%


\bibitem{ms} {\sc Macphee, I. M. and  Schuh, H. J.} (1983). A Galton-Watson branching process in varying environments with essentially constant means and two rates of growth. {\em Aust. N. Z. J. Stat.} {\bf25,} 329--338.


\bibitem{pe04} {\sc Petrov, V. V.} (2004). A generalization of the Borel-Cantelli lemma. {\it Statist. Probab. Lett.,} {\bf67}(3), 233-239.

\bibitem{roi}{\sc Roitershtein, A.} (2007). A note on multitype branching processes with immigration in a random environment. {\it Ann. Probab.}  {\bf35}(4), 1573-1592.


\bibitem{wact}{\sc Wang, H. M.} (2013). A note on multitype branching process with bounded immigration in random environment. {\it Acta Math. Sin. (Engl. Ser.)}, {\bf29}(6), 1095-1110.
\bibitem{wmp}{\sc Wang, H. M.} (2019). On the number of points skipped by a transient (1,2) random walk on the lattice of the positive half line. {\it Markov Processes Relat. Fields,}  {\bf25}, 125-148.

%

\bibitem{wy}	{\sc Wang, H. M. and Yao, H.} (2022). Two-type linear fractional branching processes in varying environments with asymptotically constant mean matrices. {\it J. Appl. Probab.,} {\bf 59}(1), 224-255.









 \end{thebibliography}
\end{document}